\documentclass[12pt]{amsart}

\usepackage{amsthm,amssymb,color,hyperref,mathtools,tikz}
\usepackage[margin=1in]{geometry}

\theoremstyle{definition}
\newtheorem{thm}{Theorem}[section]
\newtheorem{alg}[thm]{Algorithm}
\newtheorem{prop}[thm]{Proposition}

\newtheorem{example}[thm]{Example}
\newtheorem{conj}[thm]{Conjecture}

\newtheorem{lemma}[thm]{Lemma}
\newtheorem{defn}[thm]{Definition}
\newtheorem{cor}[thm]{Corollary}

\newtheorem{problem}[thm]{Problem}
\newtheorem{rmk}[thm]{Remark}

\DeclareMathOperator{\conv}{conv}
\DeclareMathOperator{\face}{face}
\DeclareMathOperator{\GL}{GL}

\DeclareMathOperator{\id}{id}
\DeclareMathOperator{\init}{in}
\DeclareMathOperator{\Inv}{Inv}

\DeclareMathOperator{\st}{st}
\DeclareMathOperator{\stell}{stell}
\DeclareMathOperator{\Newt}{Newt}

\DeclareMathOperator{\supp}{supp}

\newcommand{\calC}{\mathcal{C}}
\newcommand{\calD}{\mathcal{D}}
\newcommand{\calF}{\mathcal{F}}
\newcommand{\calG}{\mathcal{G}}

\newcommand{\calI}{\mathcal{I}}
\newcommand{\calJ}{\mathcal{J}}
\newcommand{\calL}{\mathcal{L}}
\newcommand{\calN}{\mathcal{N}}

\newcommand{\calX}{\mathcal{X}}
\newcommand{\calY}{\mathcal{Y}}
\newcommand{\calZ}{\mathcal{Z}}
\newcommand{\CC}{\mathbb{C}}
\newcommand{\iso}{\cong}
\newcommand{\NN}{\mathbb{N}}
\newcommand{\QQ}{\mathbb{Q}}
\newcommand{\RR}{\mathbb{R}}
\newcommand{\ZZ}{\mathbb{Z}}
\newcommand{\symm}{\mathfrak{S}}

\begin{document}

\title{State Polytopes Related to Two Classes of Combinatorial Neural Codes}
\author{Robert Davis}
\address{Department of Mathematics\\
         Harvey Mudd College\\
         320 E. Foothill Blvd. \\
Claremont, CA  91711}
\email{rdavis@hmc.edu}

\maketitle
	
\begin{abstract}
	Combinatorial neural codes are $0/1$ vectors that are used to model the co-firing patterns of a set of place cells in the brain.
	One wide-open problem in this area is to determine when a given code can be algorithmically drawn on the plane as a Venn diagram-like figure.
	A sufficient condition to do so is for the code to have a property called $k$-inductively pierced.
	Gross, Obatake, and Youngs recently used toric algebra to show that a code on three neurons is $1$-inductively pierced if and only if the toric ideal is trivial or generated by quadratics.
	No result is known for additional neurons in the same generality, part of the difficulty coming from the large number of codewords possible when additional neurons are used.
	
	In this article, we study two infinite classes of combinatorial neural codes in detail.
	For each code, we explicitly compute its universal Gr\"obner basis. 
	This is done for the first class by recognizing that the codewords form a Lawrence-type matrix.
	With the second class, this is done by showing that the matrix is totally unimodular.
	These computations allow one to compute the state polytopes of the corresponding toric ideals, from which all distinct initial ideals may be computed efficiently.
	Moreover, we show that the state polytopes are combinatorially equivalent to well-known polytopes: the permutohedron and the stellohedron. 
\end{abstract}


\section{Introduction}

In the $1970$s, O'Keefe et al. \cite{OKeefePlaceCells} observed that certain neurons in the brain, called \emph{place cells}, spike in their firing rates when the animal is in a particular physical location within its arena.
Figure~\ref{fig: neural activity} shows a sample set of data of place cell activity in a rat, provided in \cite[Supplementary material]{PfeifferFoster}.
Each square corresponds to the firing activity of a single place cell of a rat within a $2$ meter by $2$ meter arena.
Dark areas indicate a low firing rate while orange and yellow areas indicate a high firing rate.
The number in the upper-left corner indicates the maximum firing rate of the place cell being mapped, in Hz.

\begin{figure}
\begin{center}
\includegraphics[scale=0.4]{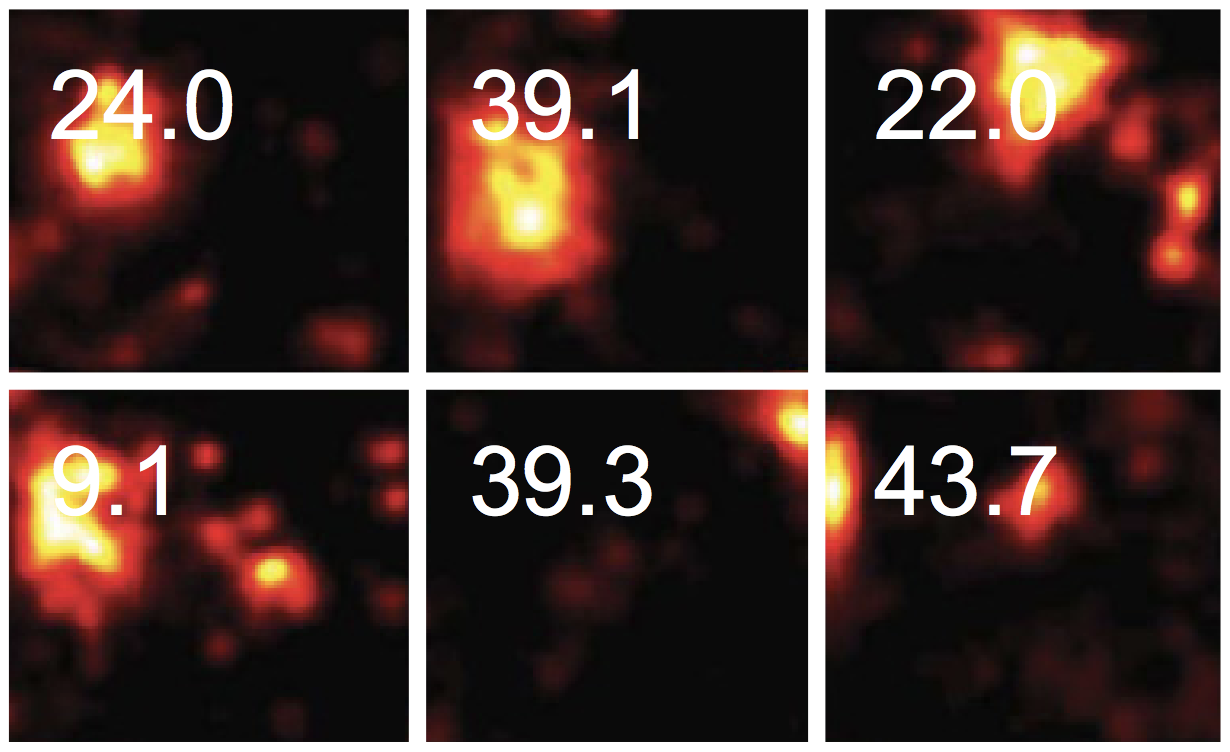}
\end{center}
\caption{Sample field data for a rat in a $2$-meter by $2$-meter arena.}\label{fig: neural activity}
\end{figure}

If a place cell is thought of as either ``active'' or ``silent,'' then one may simplify the images obtained in this way by sketching a disjoint union of simple closed curves to represent the locations in which a neuron is active.
These diagrams are called \emph{Euler diagrams}.
By labeling the curves corresponding to the $i^{th}$ neuron with $\lambda_i$, the set of co-firing patterns of the neurons can be discretized into $0/1$ vectors, where coordinate $i$ is $1$ if the animal is inside a region bounded by $\lambda_i$.
These vectors are called \emph{codewords}, and the set of all codewords for the diagram is called a \emph{combinatorial neural code on $n$ neurons}.
The regions determined by the $\lambda_i$ are often assumed to be convex, as, otherwise, it is possible to have a single codeword correspond to a region with multiple connected components.
We follow this convention throughout the paper.

Let $U_i$ denote the strict interior of $\lambda_i$.. 
The nonempty intersections of $U_1,\dots,U_n$ and their complements $U_1^c,\dots,U_n^c$ are called \emph{zones}, which can be encoded by subsets of the curves; namely, the set $Z \subseteq   \{\lambda_1,\dots,\lambda_n\}$ is a zone if $(\cap_{\lambda_i \in Z} U_i) \cap (\cap_{\lambda_j \notin Z} U_j^c)$ is nonempty. 
For simplicity, these zones can be encoded using the codeword of length $n$ where position $i$ is $0$ if the zone is contained in $U_i^c$ and position $i$ is $1$ otherwise.
In a mild abuse of notation, we will occasionally identify a codeword $c_1c_2\dots c_n$ with the vector $(c_1,c_2,\dots,c_n)$ when no confusion will arise. 

\begin{example}
	Consider the Euler diagram in Figure~\ref{fig: Euler diagram}.
	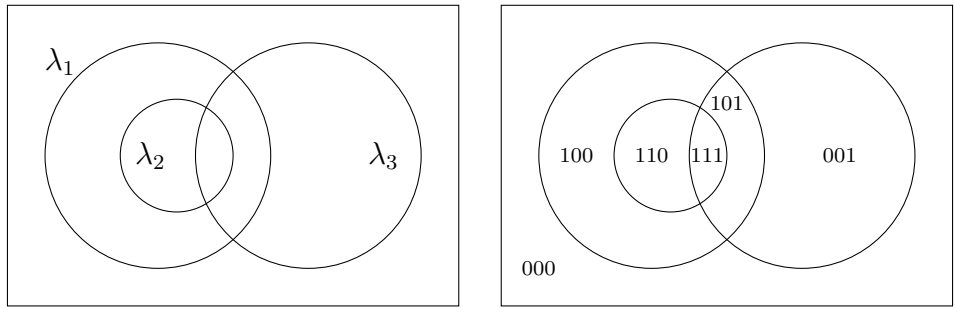
\begin{figure}
		\begin{center}
		\begin{tikzpicture}
			\draw (-3,0.5) -- (3,0.5) -- (3,4.5) -- (-3,4.5) -- cycle;
			\draw (-1,2.5) circle [radius=1.5];
			\draw (-0.75,2.5) circle [radius=.75];
			\draw (1,2.5) circle [radius=1.5];
				
			\node at (-2.3,3.75) {$\lambda_1$};
			\node at (-1.1,2.5) {$\lambda_2$};
			\node at (2,2.5) {$\lambda_3$};
		\end{tikzpicture} \quad
		\begin{tikzpicture}
		\draw (-3,0.5) -- (3,0.5) -- (3,4.5) -- (-3,4.5) -- cycle;
		\draw (-1,2.5) circle [radius=1.5];
		\draw (-0.75,2.5) circle [radius=.75];
		\draw (1,2.5) circle [radius=1.5];
		
		\node at (-2.5,1) {\tiny{$000$}};
		\node at (1.5,2.5) {\tiny{$001$}};
		\node at (-1,2.5) {\tiny{$110$}};
		\node at (-2,2.5) {\tiny{$100$}};
		\node at (-0.26,2.5) {\tiny{$111$}};
		\node at (0,3.2) {\tiny{$101$}};
	\end{tikzpicture}
		\end{center}
		\caption{An Euler diagram for a combinatorial neural code on $3$ neurons. On the left are the labels of the curves; on the right are the resulting codewords written inside their zones.}\label{fig: Euler diagram}
	\end{figure}
	There are three curves, $\lambda_1$, $\lambda_2$, and $\lambda_3$, and their interiors are $U_1$, $U_2$, and $U_3$, respectively.
	Since zones must be nonempty, are five zones: $U_1 \cap U_2^c \cap U_3^c$, whose corresponding codeword is $100$; $U_1 \cap U_2 \cap U_3^c$, whose codeword is $110$; and so forth.
	Including $000$, all six codewords are written inside of the regions to which they correspond.
\end{example}

One particularly difficult problem is to recreate the Euler diagram for a given combinatorial neural code.
If the code is \emph{$k$-inductively pierced} (see Definition~\ref{defn: inductively pierced}), then there exists an efficient algorithm for reconstructing the diagram \cite{TheoryPiercings}.
Recent progress has been made to detect this property through algebraic means \cite{NeuralRing,GrossObatakeYoungs,Hoch}.
For example, Gross, Obatake, and Youngs recently used toric algebra to show that a (well-formed) code on three neurons is $1$-inductively pierced if and only if its toric ideal is trivial or generated by quadratics \cite[Proposition 4.8]{GrossObatakeYoungs}.
No result is known for additional neurons in the same generality, with part of the difficulty coming from the large number of codewords possible when additional neurons are used.

In this article, we study two infinite classes of combinatorial neural codes in detail.
First, in Section~\ref{sec: background}, we introduce the two classes of codes under consideration, and prove necessary properties about their Euler diagrams.
In Section~\ref{sec: toric alg}, we construct toric ideals from our codes and explicitly compute their universal Gr\"obner basis. 
This is done for the first class by recognizing that the codewords form a Lawrence-type matrix.
With the second class, this is done by by showing that the matrix is totally unimodular.
Section~\ref{sec: state pols} follows by discussing a connection between initial ideals of toric ideals and polytopes.
The construction of universal Gr\"obner bases allows us to compute the state polytopes of the corresponding toric ideals, from which all distinct initial ideals may be computed efficiently.
In particular, we show that the state polytopes are combinatorially equivalent to well-known polytopes: the permutohedron and the stellohedron.
The article concludes with an open problem, a conjecture, and experimental data for other classes of combinatorial neural codes.


\section{Combinatorial Neural Codes and their Euler Diagrams}\label{sec: background}

There are two specific classes of codes that we will be analyzing throughout the paper.
The first class is the following.

\begin{defn}
	Let $n \in \ZZ_{>0}$ and let $e_i \in \RR^{n+1}$ denote the $i^{th}$ standard basis vector.
	The \emph{homogeneous star code on $n+1$ neurons} is $S_n = \{0\dots0, s_1,\dots,s_{2n}\}$ where
	\[
		s_i = \begin{cases}
				e_1 + e_{i+1} + e_{n+1} & \text{ if } 1 \leq i < n \\
				e_1 + e_{n+1} & \text{ if } i = n\\
				e_{i+1-n} + e_{n+1} & \text{ if } n < i < 2n \\
				e_{n+1} & \text{ if } i = 2n
			\end{cases}
	\]
\end{defn}

This code uses the convention that curves $\lambda_1,\dots,\lambda_n$ are always contained entirely inside of $\lambda_{n+1}$.
We use $n$ to indicate the total number of curves inside of $\lambda_{n+1}$. 

\begin{rmk}
	The reason for calling this code a \emph{star code} is because of the Euler diagram from which this code comes.
	Zones in the diagram can be represented by replacing curves $\lambda_2$ through $\lambda_n$ with vertices $v_2$ through $v_n$, and constructing edges $\{v_2,v_i\}$ for each $i = 3,\dots,n$. 
	The resulting graph is the complete bipartite graph $K_{1,n-2}$, sometimes called a \emph{star}. 
	See the left diagram in Figure~\ref{fig: diagrams} for an Euler diagram corresponding to $S_5$, along with its labeled curves and zones. 
\end{rmk}

To define our second main class of codes, first let 
\[
	A = \begin{bmatrix}
		1 & 1 & 0 \\
		0 & 1 & 1 
	\end{bmatrix}.
\]
Further, for any matrix $M \in \RR^{n \times k}$, let $\alpha(M)$ denote the $(n+1) \times k$ matrix formed by appending a single row consisting entirely of $1$s to the bottom of $M$.
Let $\oplus$ denote the direct sum operation for matrices.

\begin{defn}
	For $n \geq 1$, let $2_n = (2,\dots,2) \in \ZZ^n$.
	We let $P(2_n)$ denote the code consisting of $0\dots0$ (length $2n+1$) and the codewords formed from the $2n+1$ entries in the columns of the $2n+1 \times 3n+1$ matrix $M_n$, which is defined as follows:
	the first $3n$ columns of $M_n$ form the matrix $\alpha(A \oplus \cdots \oplus A)$ and the last column of $M_n$ is $e_{2n+1}$.
\end{defn}

An Euler diagram for $P(2_n)$ consists of a single curve $\lambda_{2n+1}$ enclosing the disjoint union of $n$ pairs of generically-intersecting circles.
See the right diagram in Figure~\ref{fig: diagrams} for an Euler diagram corresponding to $P(2_3)$. 

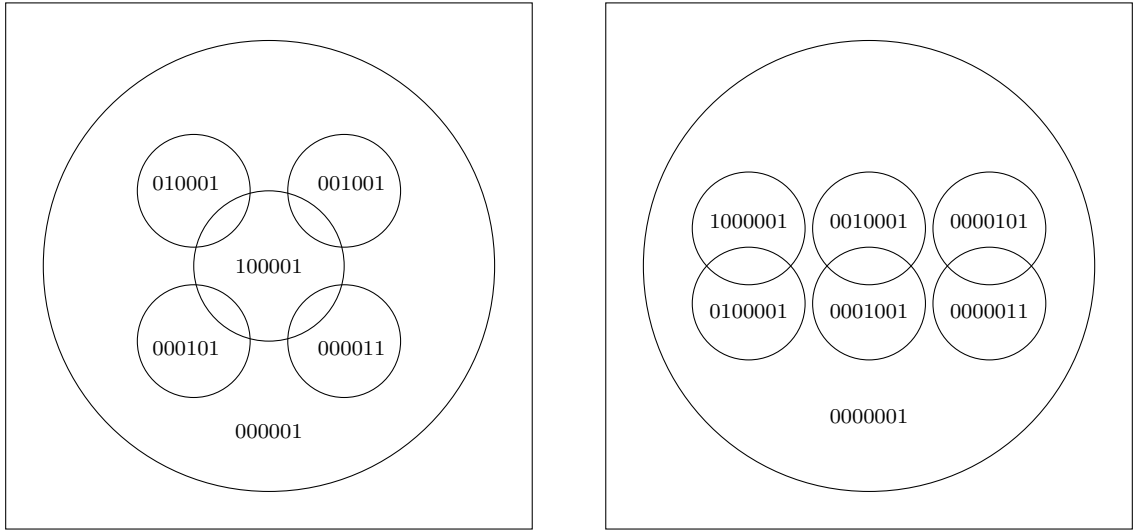
\begin{figure}
	\begin{center}
		\begin{tikzpicture}
			\draw (-3.5,-3.5) -- (3.5,-3.5) -- (3.5,3.5) -- (-3.5,3.5) -- cycle;
			\draw (0,0) circle [radius=3];
			\draw (0,0) circle [radius=1];
			\draw (-1,-1) circle [radius=0.75];
			\draw (1,1) circle [radius=0.75];
			\draw (1,-1) circle [radius=0.75];
			\draw (-1,1) circle [radius=0.75];
			
			\node at (0,0) {\tiny{$100001$}};
			\node at (-1.1,1.1) {\tiny{$010001$}};
			\node at (1.1,1.1) {\tiny{$001001$}};
			\node at (-1.1,-1.1) {\tiny{$000101$}};
			\node at (1.1,-1.1) {\tiny{$000011$}};
			\node at (0,-2.2) {\tiny{$000001$}};
		\end{tikzpicture} \qquad
		\begin{tikzpicture}
			\draw (-3.5,-3.5) -- (3.5,-3.5) -- (3.5,3.5) -- (-3.5,3.5) -- cycle;
			\draw (0,0) circle [radius=3];
			\draw (-1.6,-0.5) circle [radius=0.75];
			\draw (-1.6,0.5) circle [radius=0.75];
			\draw (0,-0.5) circle [radius=0.75];
			\draw (0,0.5) circle [radius=0.75];
			\draw (1.6,-0.5) circle [radius=0.75];
			\draw (1.6,0.5) circle [radius=0.75];
			
			\node at (-1.6,0.6) {\tiny{$1000001$}};
			\node at (-1.6,-0.6) {\tiny{$0100001$}};
			\node at (0,0.6) {\tiny{$0010001$}};
			\node at (0,-0.6) {\tiny{$0001001$}};
			\node at (1.6,0.6) {\tiny{$0000101$}};
			\node at (1.6,-0.6) {\tiny{$0000011$}};
			\node at (0,-2) {\tiny{$0000001$}};
		\end{tikzpicture}
		\end{center}
	\caption{On the left: an Euler diagram corresponding to $S_5$. On the right: an Euler diagram corresponding to $P(2_3)$. Each diagram has several codewords listed in the regions to which they correspond.}\label{fig: diagrams}
\end{figure}

\begin{rmk}
	Similarly to how we related the Euler diagrams for $S_n$ to star graphs, we can also relate $P(2_n)$ to a disjoint union of edges.
	This observation suggests a generalization of $P(2_n)$ to any $\ell  = (l_1,\dots,l_n) \in \ZZ_{\geq 0}^n$, where $P(\ell)$ now corresponds to a disjoint union of paths.
	The value $l_i$ then corresponds to the length of the $i^{th}$ component path of the graph. 
	We say more about $P(\ell)$ in Section~\ref{sec: experimental data}.
\end{rmk}

One assumption that we have been making is that the codes we are studying have Euler diagrams that are ``sufficiently generic.''
More precisely, we call an Euler diagram \emph{well-formed} if
\begin{enumerate}
	\item each curve label is used exactly once, 
	\item all curves intersect in only finitely many points,
	\item each point in the plane is passed through at most 2 times by the curves in the diagram, and
	\item each zone is connected.
\end{enumerate}

While making sketches of Euler diagrams is great for visualizing, in order to work with them more algebraically, we need to describe them in a more formal way.
An \emph{abstract description} of an Euler diagram $d$ is an ordered pair $\calD = (\calL, \calZ)$ where $\calL$ contains the labels of the curves and $\calZ \subseteq 2^{\calL}$ contains the zones of the diagram.
Since we always assume that $0\dots0$ is in a neural code, we assume that $\emptyset \in \calZ$.
An Euler diagram with abstract description $\calD$ is called a \emph{realization} or \emph{drawing} of $\calD$.
If $\calD$ is an abstract description that has a well-formed realization, then we also call $\calD$ \emph{well-formed}.

For an abstract description $\calD = (\calL,\calZ)$ and $\lambda \in \calL$, let $\calX_\lambda = \{Z \in \calZ \mid \lambda \in Z\}$.
In words, $\calX_i$ is the set of all zones containing $\lambda$. 
Further, given a zone $Z \in \calZ$ and $\Lambda \subseteq \calL$ such that $Z \cap \Lambda = \emptyset$, the \emph{$\Lambda$-cluster of $Z$}, denoted $\calY_{Z, \Lambda}$ is
\[
	\calY_{Z,\Lambda} = \{Z \cup \Lambda_i \mid \Lambda_i \subseteq \Lambda\}.
\]
Notice that $\calY_{Z,\Lambda}$ is not always a set of zones in $\calD$, so something interesting is happening when $\calY_{Z,\Lambda} \subseteq \calZ$.
	
\begin{defn}
	Let $\calD = (\calL,\calZ)$ be an abstract description, and let $\Lambda =\{\lambda_1,\dots,\lambda_k\}  \subseteq \calL$ be distinct labels.
	Say that $\lambda_{k+1} \in \calL$ is a \emph{$k$-piercing of $\Lambda$ in $\calD$} if there exists a zone $Z\in \calZ$ for which
	\begin{enumerate}
		\item $\lambda_i \notin Z$ for each $i \leq k+1$,
		\item $\calX_{\lambda_{k+1}} = \calY_{Z \cup \{\lambda_{k+1}\}, \Lambda}$, and
		\item $\calY_{Z,\Lambda} \subseteq \calZ$.
	\end{enumerate}
	When this happens, we say that $\lambda_{k+1}$ is a $k$-piercing \emph{identified by the background zone $Z$}.
\end{defn}

Given $\calD = (\calL, \calZ)$ and $\lambda \in \calL$, define the \emph{removal of $\lambda$ from $\calD$} to be
\[
	\calD \setminus \lambda = (\calL \setminus \{\lambda\}, \{Z \setminus \{\lambda\} \mid Z \in \calZ\}\}).
\]
This is the abstraction of deleting a neuron from a neural code $\calC$:
\[
	\calC \setminus \lambda = \{(c_1,\dots,c_{\lambda -1}, c_{\lambda+1},\dots,c_n) \mid (c_1,\dots,c_n) \in \calC\}.
\]
We now give another main definition regarding combinatorial neural codes.

\begin{defn}\label{defn: inductively pierced}
	The abstract description $\calD = (\calL,\calZ)$ is \emph{$k$-inductively pierced} if $\calD$ has a $j$-piercing $\lambda$ for some 
	$j = 0,\dots,k$ such that $\calD \setminus \lambda$ is $k$-inductively pierced.
	To establish a base for verification, declare the abstract description $(\emptyset,\{\emptyset\})$ to be $k$-inductively pierced for all $k$.
	Say a code $\calC$ is \emph{$k$-inductively pierced} if it has a well-formed abstract description that is $k$-inductively pierced.
\end{defn}

Note that if $\calD$ is $k$-inductively pierced, then it is automatically $l$-inductively pierced for all $l \geq k$. 
The codes $S_n$ and $P(2_n)$ have canonical abstract descriptions, where the curve labels and zones are given in their definitions.
In this instance, we will use $\calL$ and $\calZ$ to denote their respective sets of curve labels and zones.

It is generally difficult to identify when a code is $k$-inductively pierced for large $k$, but when $k=0$, there is a simple description.

\begin{prop}[{\cite[Proposition~2.7]{GrossObatakeYoungs}}]\label{prop: 0 piercings}
	A well-formed abstract description $\calD$ is inductively $0$-pierced if and only if no two curves in any well-formed realization of $\calD$ intersect.
\end{prop}

We are now able to prove our first result.

\begin{prop}\label{prop: S_n ind pierced}
	The codes $S_n$ and $P(2_n)$ are codes for $1$-inductively pierced for all $n$.
\end{prop}

\begin{proof}
	This will follow by induction and a straightforward application of the definition of $1$-inductively pierced.
	We first consider $S_n$, and use $n=1$ as the initial case, so that $S_1 = \{00,11,01\}$.
	This has a well-formed realization as the Euler diagram consisting of one curve completely contained inside another curve.
	By Proposition~\ref{prop: 0 piercings}, $S_1$ is $0$-inductively pierced.
	Thus, $S_1$ is $1$-inductively pierced.
	
	Now suppose $n > 1$, set $\Lambda = \{\lambda_1\}$ and consider curve $\lambda_n$. 
	Then $\lambda_n$ is a $1$-piercing of $S_n$ identified by the background zone $Z = \{\lambda_{n+1}\}$:
	observe that
	\[
		\calX_{\lambda_n} =  \{\{\lambda_1,\lambda_n,\lambda_{n+1}\},\{\lambda_n,\lambda_{n+1}\}\},
	\]
	and
	\[
		\calY_{\{\lambda_n,\lambda_{n+1}\},\{\lambda_1\}} = \{ \{\lambda_1, \lambda_n, \lambda_{n+1}\}, \{\lambda_n,\lambda_{n+1}\}\} = \calX_{\lambda_n}.
	\]
	Moreover, $\lambda_n \notin Z$, and
	\[
		\calY_{\{\lambda_{n+1}\},\{\lambda_1\}} = \{\{\lambda_{n+1}\},\{\lambda_1,\lambda_{n+1}\}\} \subseteq \calZ.
	\]
	Therefore, $\lambda_n$ is a $1$-piercing of $S_n$.
	Since $S_n \setminus \{\lambda_n\} = S_{n-1}$ and $S_{n-1}$ is $1$-inductively pierced, $S_n$ is $1$-inductively pierced, as claimed.

	An inductive argument proves that $P(2_n)$ is $1$-inductively pierced as well.
	In this case, setting $\Lambda = \{\lambda_{2n-1}\}$, the reader can verify that $\lambda_{2n}$ is a $1$-piercing identified by background zone $\{\lambda_{n+1}\}$.
	Then, $\lambda_{2n-1}$ is a $0$-piercing of $P(2_n) \setminus \{\lambda_{2n}\}$ identified by the same background zone.
	Finally, $(P(2_n) \setminus \{\lambda_{2n}\}) \setminus \{\lambda_{2n-1}\} = P(2_{n-1})$, which is $1$-inductively pierced.
	Therefore, $P(2_n)$ is $1$-inductively pierced.
\end{proof}

Although Proposition~\ref{prop: S_n ind pierced} is only a small result, it has broader algebraic implications, as we will see in the next section.


\section{Toric Ideals and the Gr\"obner Fan}\label{sec: toric alg}

An essential component to our work will be to relate $S_n$ and $P(2_n)$ to polynomials.
In this section we give the necessary algebraic background, and will draw the connections in Section~\ref{sec: state pols}.
For more information about the topics presented in this section, see, for example, \cite{sturmfels}.


\subsection{Monomial Orders on Polynomials}

Given a semigroup $A \subseteq \ZZ^n$, the \emph{semigroup algebra generated by $A$} is 
\[
	\CC[A] := \CC[x^a \mid a \in A] \subseteq \CC[x_1^{\pm 1},\dots,x_n^{\pm 1}],
\]
where, if $a = (a_1,\dots,a_n)$, then $x^a := x_1^{a_1}x_2^{a_2}\cdots x_n^{a_n}$.
In order to get more control over semigroup algebras, we want to keep track of how their generators produce elements.
We do so in the following way: again suppose $A = \{b_1,\dots, b_m\} \subseteq \ZZ^n$, and $A$ is a set of minimal generators of $\NN A$, the set of all nonnegative integer linear combinations of elements in $A$.
For each $b_i \in A$ let us associate a monomial $x^{b_i} = x_1^{b_{i,1}}\dots x_n^{b_{i,n}} \in \CC[\NN A]$ and a variable $t_i \in \CC[t_1,\dots,t_m]$. 
We therefore have the homomorphism
\[
	\pi_A: \CC[t_1,\dots,t_m] \to \CC[\NN A]
\]
defined by $\pi_A(t_i) = x^{b_i}$ 
By construction $\pi_A$ is surjective, hence, by the First Isomorphism Theorem,
\[
	\CC[t_1,\dots,t_m]/I_A \iso \CC[\NN A].
\]
where $I_A = \ker \pi_A$.
The ideal $I_A$ is called the \emph{toric ideal} of $A$, and contains an astounding amount of information about $A$.

\begin{rmk}
	Toric ideals are also commonly introduced by starting with a matrix $M \in \ZZ^{n \times k}$ and defining $I_M$ to be the toric ideal of the set consisting of the columns of $M$, treated as vectors in $\RR^n$.
	This construction is slightly more general, as a matrix may have repeated columns.
	For this article, however, we will never encounter a matrix with repeated columns.
	Defining toric ideals through matrices will occasionally be helpful, and we will occasionally treat our codes as both matrices and as sets of columns of matrices.
\end{rmk}

One useful result we can state is the following.

\begin{lemma}[{\cite[Lemma 4.14]{sturmfels}}]\label{lem: homog toric}
	Let $A \in \ZZ^{n \times k}$.
	The ideal $I_A$ is homogeneous if and only if there is some vector $w \in \QQ^n$ for which $a \cdot w = 1$ for each column $a$ of $A$.
\end{lemma}

By construction, both $S_n$ and $P(2_n)$ have homogeneous toric ideals: each nonzero codeword $a$ satisfies $a \cdot e_{n+1} = 1$ when $a \in S_n$ and $a \cdot e_{2n+1} = 1$ when $a \in P(2_n)$. 
This explains the use of the word ``homogeneous'' when introducing $S_n$.

Next, if $\prec$ is a total ordering on monomials, then each polynomial $f \in \CC[t_1,\dots,t_m]$ has a unique \emph{initial term}, denoted $\init_{\prec}(f)$, which is greatest among the monomials of $f$.
This further leads to the \emph{initial ideal} of an ideal $I$, defined as
\[
	\init_{\prec}(I) = (\init_{\prec}(f) \mid f \in I).
\]
Although $\init_{\prec}(I)$, as defined, is generated by an infinite number of polynomials, it is still an ideal of $\CC[t_1,\dots,t_m]$, so by Hilbert's Basis Theorem, it has a finite generating set. 

A class of monomial orders that will be important for us is that of weight orders.
Let $w \in \RR^m$ and let $\sigma$ denote some monomial order.
The \emph{weight order} $\prec_{w,\sigma}$ determined by $w$ and $\sigma$ sets $t^a \prec_{w,\sigma} t^b$ if and only if $w\cdot a < w \cdot b$, where $\cdot$ is the ordinary dot product.
More generally, let $W = \{w_1,\dots,w_m\} \subseteq \RR^m$ be linearly independent.
The \emph{weight order} $\prec_W$ determined by $W$ is defined by setting $t^a \prec_W t^b$ if there is some $1 \leq i \leq m$ for which $w_j\cdot a = w_j \cdot b$  for each $j < i$ and $w_i\cdot a < w_i \cdot b$.
It is well known that all monomial orders can be represented by an appropriate choice of weight vectors $W$ \cite[Ch. 2, \S4]{CoxLittleOShea}.

Now, if $I = (g_1,\dots,g_k)$ is an ideal of polynomials, it is not necessarily true that the ideal $(\init_{\prec}(g_1),\dots,\init_{\prec}(g_k))$ equals $\init_{\prec}(I)$.
However, if this does occur, then we call $\{g_1,\dots,g_k\}$ a \emph{Gr\"obner basis of $I$ with respect to $\prec$}. 
A Gr\"obner basis $\calG$ is \emph{reduced} if the leading coefficient (with respect to $\prec$) of every element is $1$ and if, for every $g,g' \in \calG$, $\init_{\prec}(g)$ does not divide any term of $g'$.
While there are many Gr\"obner bases of an ideal, there is a unique reduced Gr\"obner basis with respect to each term order \cite[Ch. 2, \S7, Theorem 5]{CoxLittleOShea}.
A \emph{universal Gr\"obner basis} of $I$ is a Gr\"obner basis for every monomial order.
Such a basis always exists, since one may take the union of all reduced Gr\"obner bases of $I$.
This particular set we call \emph{the universal Gr\"obner basis} of $I$.

For a combinatorial neural code $\calC$, its toric ideal is denoted $I_{\calC}$  and is defined as 
\[
	I_{\calC} = \ker \pi_{\calC \setminus \{0\dots0\}}.
\]
The reason that the all-zeros code is excluded is because, otherwise, the quotient ring $\CC[t_1,\dots,t_m]/\ker \pi_{\calC}$ would always identify $1$ with $t_m$.
So, including $0\dots0$ tells us nothing unique to $\calC$, and is excluded from the start.
We emphasize that it is important to recognize the distinction between $I_{\calC}$ where $\calC$ is a code and $I_A$ where $A$ is a set or matrix: in the former, the zero codeword is excluded when constructing the toric ideal, and in the latter, all vectors in $A$ (or columns of $A$) are included whether or not any of them are the zero vector.

We are now ready to state several results that relate combinatorial neural codes and toric ideals.

\begin{thm}[see \cite{GrossObatakeYoungs}]\label{thm: well-formed properties}
	Let $\calC$ be a well-formed code on $n$ neurons such that each curve is contained in some zone. 
	\begin{enumerate}
		\item The toric ideal $I_{\calC} = (0)$ if and only if $\calC$ is $0$-inductively pierced.
		\item If $\calC$ is $1$-inductively pierced then the toric ideal $I_{\calC}$ is either generated by quadratics or $I_{\calC} = (0)$.
		\item When $n=3$, the code is $1$-inductively pierced if and only if the reduced Gr\"obner basis of $I_{\calC}$ with respect to the weighted grevlex order with the weight vector $w = (0, 0, 0, 1, 1, 1, 0)$ contains only binomials of degree $2$ or less.
	\end{enumerate}
\end{thm}

Notice that although there is a complete characterization of $0$-inductively pierced codes in terms of the toric ideal, there is no such characterization for $1$-inductively pierced codes for more than $3$ curves.
If $n$ total curves are used to construct a combinatorial neural code, then there are most $2^n-1$ zones, which is achieved when all possible intersections of curves occur. 
So, we expect it to be computationally challenging to extend Theorem~\ref{thm: well-formed properties} to large numbers of curves.


\subsection{Encoding Initial Ideals via Polyhedra}

If $w,w'$ are two weights on the monomials in $\CC[t_1,\dots,t_m]$ it is possible that $\init_{\prec_w}(I) = \init_{\prec_{w'}}(I)$.
When this happens, we say that $w$ and $w'$ are \emph{equivalent}. 
Given an ideal $I$, the set of weight vectors equivalent to a given $w \in \RR^m$ forms a convex polyhedral cone, and the set of all cones is called the \emph{Gr\"obner fan} of $I$.
When $I$ is \emph{homogeneous} that is, when $I$ is generated by homogeneous polynomials, their union is all of $\RR^m$. 
So, the top-dimensional cones in the Gr\"obner fan of $I$ correspond to the distinct initial ideals of $I$. 
If one can determine a representative in each of these cones, then the work needed to identify the unique initial ideals is significantly reduced.
This process is made easier by introducing polyhedra.

A \emph{polytope}, $P \subseteq \RR^m$, is the convex hull of finitely many points $v_1,\dots,v_k \in \RR^m$, and is written
\[
	P = \conv\{v_1,\dots,v_k\} = \left\{ \sum_{i=1}^k \lambda_iv_i \mid \lambda_1,\dots,\lambda_k \in \RR_{\geq 0}, \sum_{i=1}^k \lambda_i = 1,\, v_i \in \RR^n \text{ for all } i\right\}.
\]
We say that $P$ is \emph{lattice} if its vertices are a subset of $\ZZ^m$.
We could equivalently defined $P$ as the intersection of finitely many halfspaces in $\RR^m$, i.e.
\[
	\bigcap_{i=1}^l \{x = (x_1,\dots,x_m) \mid a_i\cdot x \leq b_i\}
\]
for some choices of $a_1,\dots,a_l \in \RR^m$ and $b_1,\dots,b_l \in \RR$.

A common way to translate between polytopes and polynomials is the following.
Let 
\[
	p(x_1,\dots,x_m) = \sum_{a \in \ZZ^m} c_ax^a \in \CC[x_1^{\pm 1},\dots,x_m^{\pm 1}]
\]
be a Laurent polynomial, that is, $c_a \in \CC$ for all $a \in \ZZ^m$ and $c_a = 0$ for all but finitely many $a \in \ZZ^m$.
The \emph{Newton polytope of $p$} is
\[
	\Newt(p) = \conv\{a \in \ZZ^m \mid c_a \neq 0\}.
\]
Newton polytopes occur in various algebraic settings, such as when bounding the number of isolated complex roots of polynomial systems \cite{Bernstein,Khovanskii,Kushnirenko}, or when solving polynomial systems via homotopy continuation \cite{HuberSturmfels}.

\begin{example}
	A lattice polytope that will be important for us is the \emph{permutohedron}, defined by
	\[
		\Pi_n = \conv\{(\pi_1,\dots,\pi_n) \mid \pi_1\dots\pi_n \in \symm_n\}
	\]
	where $\symm_n$ denotes the symmetric group on $[n]$.
	This polytope is well-known to be $(n-1)$-dimensional, each of its points satisfying $\sum x_i = \binom{n+1}{2}$.
	It can also be described as the Minkowski sum
	\[
		\Pi_n = \{(1,\dots,1)\} + \sum_{1 \leq i < j \leq n} [e_i,e_j]
	\]
	where $[v,w]$ denotes the line segment between $v, w \in \RR^n$.
	This description implies
	$\Pi_n = \Newt(x_1\cdots x_n\det V(x_1,\dots,x_n))$ where $V(x_1,\dots,x_n)$ denotes the Vandermonde matrix in the variables $x_1,\dots,x_n$.
\end{example}

A \emph{face} of the polytope $P \subseteq \RR^m$ is any subset $F\subseteq P$ of the form
\[
	F = P \cap \{x \in \RR^m \mid a \cdot x = b\}
\]
where $a \in \RR^m$ and $b \in \RR$ satisfy $a\cdot x \leq b$ for all $x \in P$.
Now, one may instead choose $a \in \RR^m$ and use it to identify a face of $P$: set
\[
	\face_a(P) = \{x \in P \mid a\cdot v \leq a \cdot x \text{ for all } v \in P\}.
\]
While the points of $\RR^m$ are not in bijective correspondence with the faces of $P$, there is a helpful way to partition $\RR^m$ into equivalence classes according to faces of $P$.

Let $F$ be a face of $P$.
The \emph{normal cone of $F$ at $P$} is
\[
	N_P(F) = \{ a \in \RR^m \mid \face_a(P) = F\}.
\]
The \emph{normal fan of $P$} is
\[
	N(P) = \bigcup_{F \text{ a face of } P} N_P(F).
\]
If $P$ is full-dimensional, then the $k$-dimensional cones of $N(P)$ are in bijection with the $(\dim P - k)$-dimensional faces of $P$. 

\begin{example}\label{ex: normal fan of permuto}
	The normal fan of $\Pi_n$ is well-known to be constructed as follows:
	first, since $\Pi_n$ lies on an affine hyperplane orthogonal to $(1,\dots,1) \in \RR^n$, every cone in $N(\Pi_n)$ contains the subspace $\RR(1,\dots,1)$. 
	Let $N(\Pi_n) / \RR(1,\dots,1)$ denote the set of normal cones in $N(\Pi_n)$ reduced by $\RR(1,\dots,1)$.
	The maximal cones of $N(\Pi_n) / \RR(1,\dots,1)$ are the set of \emph{Weyl chambers}
	\[
		C_{\pi} = \{ (x_1,\dots,x_n) \mid x_{\pi_1} \leq \cdots \leq x_{\pi_n}\},
	\]
	where $\pi = \pi_1\dots\pi_n \in \symm_n$.
	Ranging over all $\pi \in \symm_n$, one obtains $N(\Pi_n)$.
	Notice that $(\pi_1,\dots,\pi_n) \in C_{\pi}$ for all $\pi$.
	Figure~\ref{fig: N(P_3)} is a sketch of $N(\Pi_3)/\RR(1,\dots,1)$ with $\Pi_3$ overlaid and its vertices labeled.
	The normal fan consists of six three-dimensional cones, six two-dimensional cones, and two one-dimensional cones.
	
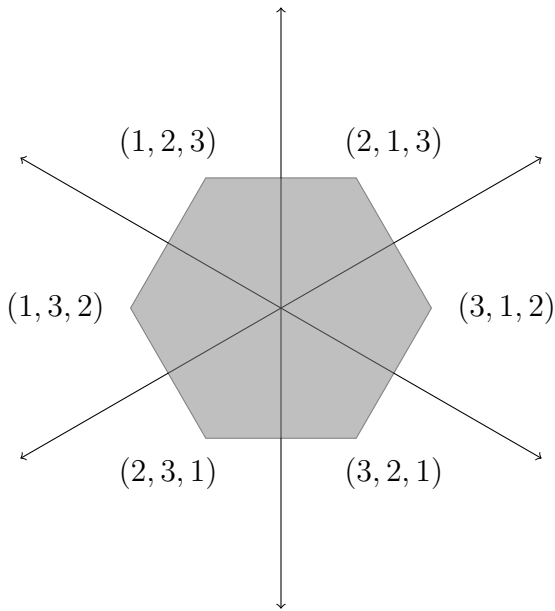
\begin{figure}
	\begin{tikzpicture}
	\draw[<->] (0,-4) -- (0,4);
	\draw[<->] (-3.464,2) -- (3.464,-2);
	\draw[<->] (3.464,2) -- (-3.464,-2);
	
	\draw[fill=gray,opacity=0.5] (0:2) \foreach \x in {60,120,...,359} {
            -- (\x:2)
        } -- cycle (90:2);
        
        \node at (-1.5,2.2) {$(1,2,3)$};
        \node at (1.5,2.2) {$(2,1,3)$};
        \node at (-3,0) {$(1,3,2)$};
        \node at (3,0) {$(3,1,2)$};
        \node at (-1.5,-2.2) {$(2,3,1)$};
        \node at (1.5,-2.2) {$(3,2,1)$};
        
	\end{tikzpicture}
	\caption{The permutohedron $\Pi_3$ overlaid with its normal fan, both viewed from $(1,1,1)$.}\label{fig: N(P_3)}
\end{figure}
\end{example}

The following result is what we will use to connect polytopes and initial ideals of homogeneous ideals. 

\begin{thm}[{\cite[Theorem 2.5]{sturmfels}}]\label{thm: state polytope}
	Let $I$ be a homogeneous ideal in $\CC[t_1,\dots,t_m]$.
	There exists a (lattice) polytope whose normal fan is exactly the Gr\"obner fan of $I$.
\end{thm}

Any polytope satisfying the conclusion of Theorem~\ref{thm: state polytope} is called a \emph{state polytope} for $I$.
The next section will be dedicated to constructing the state polytopes for $S_n$ and $P(2_n)$.


\section{Constructing the State Polytopes}\label{sec: state pols}


\subsection{For the code $S_n$}

Let $t^a - t^b \in I_A$ for some set of vectors $A \subseteq \ZZ^n$.
We call $t^a - t^b$ \emph{primitive} if there is no $t^v - t^w \in I_A$ for which $t^v$ divides $t^a$ and $t^w$ divides $t^b$.
The set of primitive binomials in $I_A$ is called the \emph{Graver basis} of $I_A$.
The Graver basis of an ideal always contains the universal Gr\"obner basis, but it is rare for the two sets to coincide.
Under certain conditions, though, equality can occur.
In order to state one such condition, recall that if $M$ is an $n \times k$ matrix, then the \emph{Lawrence lifting of $M$} is the matrix
\[
	\Lambda(M) = \begin{bmatrix}
		M & 0 \\
		\id_k & \id_k
	\end{bmatrix}
\]
where $\id_k$ is the $k \times k$ identity matrix and $0$ is the $n \times k$ zero matrix. 

\begin{thm}[{\cite[Theorem~7.1]{sturmfels}}]\label{thm: minimal binomials form universal gb}
	Let $M$ be any finite generating set of $\NN M$.
	The following sets are identical:
	\begin{enumerate}
		\item the universal Gr\"obner basis of $I_{\Lambda(M)}$,
		\item the Graver basis of $I_{\Lambda(M)}$,
		\item the minimal binomial generating set of $I_{\Lambda(M)}$.
	\end{enumerate}
\end{thm}

Fortunately, the previous theorem is enough for us to construct the universal Gr\"obner basis of $S_n$.
Recall that if $B \in \GL_n(\ZZ)$, then we say the transformation $x \mapsto Bx$ is \emph{unimodular}.
Additionally, if $A \in \ZZ^{n \times k}$, then $I_A = I_{BA}$, since multiplying $A$ by $B$ on the left does not affect linear dependencies of the resulting columns. 

\begin{prop}\label{prop: ugb}
	A Gr\"obner basis for $I_{S_n}$ is
	\[
		U_n = \{t_it_{n+j} - t_jt_{n+i} \mid 1 \leq i < j \leq n\}.
	\]
\end{prop}

\begin{proof}
	Let $A$ denote the $n+1 \times 2n$ matrix for which column $i$ is $s_i$, and let $A_i$ denote row $i$ of $A$.
	The transformation $f$ which replaces $A_{n+1}$ with $A_{n+1} - A_2 - \dots - A_n$ is a unimodular transformation, so $f(A)$ has the same toric ideal as $A$.
	Moreover, 
	\[
		f(A) = \Lambda(\begin{bmatrix} 1 & \cdots & 1 \end{bmatrix}),
	\]
	that is, $f(A)$ is the Lawrence lifting of the $1 \times n$ matrix for which all entries are $1$. 
	By \cite[Theorem~7.1]{sturmfels}, any reduced Gr\"obner basis of $I_{f(A)}$ is the universal Gr\"obner basis of $I_{f(A)}$. 
	Since $I_{f(A)} = I_A$, any reduced Gr\"obner basis of $I_{f(A)}$ is the universal Gr\"obner basis for $I_A$, and therefore also for $I_{S_n}$. 
	
	By Proposition~\ref{prop: S_n ind pierced} and Theorem~\ref{thm: well-formed properties}, part $2$, $I_{S_n}$ is either trivial or generated by quadratics.
	Since $I_{S_n}$ is homogeneous, it follows from Theorem~\ref{thm: minimal binomials form universal gb} that if we can show $U_n$ is a minimal generating set for $I_{S_n}$, then $U_n$ is the universal Gr\"obner basis of $I_{S_n}$.
	
	It is clear that the binomials in $U_n$ are primitive, so these must all be part of the Graver basis of $I_{f(A)}$.
	So, they are part of a minimal binomial generating set of $I_{f(A)}$.
	It suffices to show that no other homogeneous quadratic binomial is part of a minimal generating set of $I_{f(A)}$.
	
	Suppose $t_it_j - t_kt_l$ is a generator of $I_{f(A)}$ that is not in $U_n$, and without loss of generality assume $i < j$. 
	First assume $i,j \leq n$.
	Then $\pi(t_it_j)$ is divisible by $x_1^2$, implying that $k,l \leq n$.
	However, this forces $\{k,l\} = \{i,j\}$ since $\pi(t_it_j)$ and $\pi(t_kt_l)$ are also both divisible by $x_{i+1}x_{j+1}$, and there is only one such possibility since $i,j \leq n$.
	So, $t_it_j = t_kt_l$, which contradicts the assumption that $g$ is a generator of $I_{f(A)}$.
	A similar contradiction occurs if we assume $n < i,j$.
	Therefore, $i \leq n < j$. 
	It is clear that, in this case, there is a unique choice of $k,l$ for which $t_it_j-t_kt_l \in I_{f(A)}$, and this choice is an element of $U_n$.
	Therefore, $U_n$ is a minimal generating set of $I_{f(A)}$, and is the universal Gr\"obner basis of $I_{S_n}$.	
\end{proof}

There are various algorithms for constructing state polytopes of $I$ given known information about the ideal, some of which can be rather inefficient.
However, since we have already identified the universal Gr\"obner basis for homogeneous star codes, we can invoke a more efficient method.
Given an ideal $I \subseteq \CC[t_1,\dots,t_m]$ generated by monomials, let $\sum I_d$ denote the sum of all vectors $a \in \NN^m$ for which $x^a$ has degree $d$ and $t^a \in I$.

\begin{alg}[{\cite[Algorithm~3.5]{sturmfels}}]\label{alg: state pol}
	Given the universal Gr\"obner basis $U$ of a homogeneous ideal $I$ in $\CC[t_1,\dots,t_m]$, the following algorithm produces a state polytope for $I$:
	\begin{enumerate}
		\item Let $D = \max\{\deg(t^a - t^b) \mid t^a - t^b \in U\}$.
		\item Set
			\[
				\Newt(U) = \sum_{p \in U} \Newt(p).
			\]
		\item For each vertex $v$ of $\Newt(U)$, do the following:
			\begin{enumerate}
				\item Select a vector $w \in N_{\Newt(U)}(\{v\})$ arbitrarily.
				\item Determine the initial ideal $\init_{\prec_w}(I) = (\init_{\prec_c}(p) \mid p \in U)$.
				\item Compute the point $\sum_{d=1}^D \init_{\prec_w}(I)_d$.
					This is a vertex of the state polytope; if it has not been previously output, do so now. 
			\end{enumerate}
	\end{enumerate}
\end{alg}

It is apparent that some portions of Algorithm~\ref{alg: state pol} simplify right away; namely, since $U_n$ consists of quadratic binomials, part $3$(c) will output a sum of vectors of the form $e_i + e_j$, $i \neq j$. 

Two polytopes $P,P' \subseteq \RR^m$ are \emph{unimodularly equivalent} if there is some matrix $M \in \GL_m(\ZZ)$ and some vector $v \in \ZZ^m$ for which $P'$ is the image of $P$ under $x \mapsto Mx + v$.
If $P \subseteq \RR^n$ and $P' \subseteq \RR^m$ with $n > m$, then we also say $P$ and $P'$ are unimodularly equivalent if $P$ is unimodularly equivalent to $P' \times \{0_{n-m}\}$. 
Unimodular equivalence is stronger than an injective image; unimodular equivalence also preserves (relative) volumes, lattice point data, and combinatorial equivalence

\begin{prop}\label{prop: newt is permuto}
	The polytope $\Newt(U_n)$ is unimodularly equivalent to $\Pi_n$.
\end{prop}

\begin{proof}
	The Newton polytope of $U_n$ is
	\[
		\Newt(U_n) = \sum_{1 \leq i < j \leq n} [e_i + e_{n+j}, e_j + e_{n+i}].
	\]
	Each vertex of a Minkowski sum is the sum of vertices of the Minkowski summands.
	By restricting to the first $n$ coordinates, $\Newt(U_n)$ projects to $\Pi_n$.
	In our case, for each vertex $v$ of $\Newt(U_n)$, whenever $e_i + e_{n+j}$ is selected to construct $v$, there is an instance of $e_j + e_{n+i}$ that is not chosen to contribute.
	In particular, if an endpoint involving $e_i$ is chosen $k$ times, then the endpoint involving $e_{n+k}$ is unchosen $k$ times, that is, $e_{n+k}$ is chosen $n-1-k$ times. 
	So, each vertex of $\Newt(U_n)$ satisfies the equations $x_i + x_{n+i} = n - 1$.
	
	Let $f:\RR^{2n} \to \RR^{2n}$ be defined by $f(x) = Lx - v$, where $v = e_{n+1} + \cdots + e_{2n}$ and $L = (l_{i,j})$ is the matrix
	\[
		l_{i,j} = \begin{cases}
				1 & \text{ if } i=j \text{ or } i = j+n \\
				0 & \text{ else }
			\end{cases}
	\]
	It is clear that $f$ fixes $\Newt(U_n)$ in the first $n$ coordinates.
	In coordinate $k \in \{n+1,\dots,2n\}$ of $w = (w_1,\dots,w_{2n}) \in \Newt(U_n)$, $w_k$ is sent to $w_{k-n} + w_k - (n-1) = (n-1) - (n-1) = 0$. 
	So, 
	\[\begin{aligned}
		f(\Newt(U_n)) &= f(\sum_{1 \leq i < j \leq n} [e_i + e_{n+j}, e_j + e_{n+i}]) \\
				&= \sum_{1 \leq i < j \leq n} [e_i, e_j] \\
				&= (\Pi_n - (1,\dots,1)) \times \{0_n\}.
	\end{aligned}\]
	Since $v$ is a lattice point and $\det L = 1$, $f$ is a unimodular transformation.
	Therefore, $\Newt(U_n)$ and $\Pi_n$ are unimodularly equivalent.
\end{proof}

An explicit description of the vertices for $\Newt(U_n)$ follows immediately.
Given $\pi = a_1\dots a_n \in \symm_n$, let $\pi^c = (n+1-a_1)(n+1-a_2)\dots(n+1-a_n)$ denote the \emph{complement} of $\pi$. 

\begin{cor}\label{cor: vertices of state pol}
	The vertices of $\Newt(U_n)$ are of the form $(\pi,\pi^c) - (1,\dots,1)$ where $\pi \in \symm_n$. 
\end{cor}

We also have the following useful lemma.

\begin{lemma}\label{lem: vertices in Weyl chambers}
	For each vertex $v$ of $\Newt(U_n)$, $v \in N_{\Newt(U_n)}(\{v\})$. 
\end{lemma}

\begin{proof}
	Let $0_n = (0,\dots,0) \in \RR^n$, and treat permutations in $\symm_n$ as vectors of $\RR^n$.
	As seen in Example~\ref{ex: normal fan of permuto}, $\pi \in N_{\Pi_n}(v)$.
	So, $\pi \times \{0_n\} \in N_{\Pi_n \times 0_n}(\pi \times 0_n)$. 
	Following the unimodular transformation from the previous proposition, $f^{-1}(\pi \times 0_n) = (\pi,\pi^c) - (1,\dots,1)$.  
\end{proof}

Describing the initial ideals of $I_{S_n}$ will be made easiest by introducing one more definition.
Let $\pi \in \symm_n$.
The \emph{inversion set} of $\pi$ is
\[
	\Inv(\pi) = \{(i,j) \in [n]^2 \mid i < j,\, \pi_i > \pi_i\},
\]
and an element of $\Inv(\pi)$ is called an \emph{inversion} of $\pi$.

\begin{lemma}\label{lem: init ideals}
	Let $v = (\pi,\pi^c)$ for some $\pi \in \symm_n$.
	If $i < j$, then
	\[
		\init_{\prec_v}(t_it_{n+j} - t_jt_{n+i}) = \begin{cases}
				t_it_{n+j} & \text{ if } \pi_i > \pi_j \\
				t_jt_{n+1}& \text{ if } \pi_j > \pi_i \\
			\end{cases}.
	\]	
	Consequently, 
	\[
		\init_{\prec_v}(I) = (t_it_{n+j} \mid (i,j) \in \Inv(\pi))
	\]
	and these initial ideals, ranging over all vertices $v$ of $\Newt(U_n)$, are pairwise distinct.
\end{lemma}

\begin{proof}
	First note that if $\pi_i > \pi_j$, then $\pi_{n+j} > \pi_{n+i}$.
	Adding corresponding sides of the inequalities, we get $\pi_i + \pi_{n+j} > \pi_j + \pi_{n+i}$, which implies $t_it_{n+j} \succ_v t_jt_{n+i}$.
	Similarly, if $\pi_i < \pi_j$, then $t_it_{n+j} \succ_v t_jt_{n+i}$.
	This proves the first part of the lemma.
	The last claim follows from ranging over all $\pi \in \symm_n$ and recognizing that the initial ideals are pairwise distinct since each $(\pi,\pi^c)$ lies in a distinct maximal cone of $N(\Newt(U_n))$. 
\end{proof}

\begin{thm}
	The state polytope of $I_{S_n}$ is unimodularly equivalent to $\Pi_n$.
	In particular, there are exactly $n!$ initial ideals of $I_{S_n}$, which can be identified by the weight vectors $(\pi,\pi^c)$, ranging over all $\pi \in \symm_n$. 
\end{thm}

\begin{proof}
	We will use Algorithm~\ref{alg: state pol} and exploit our previous results.
	Here, we have
	\[
		D = \max\{\deg(t^a - t^b) \mid t^a - t^b \in U_n\} = 2.
	\]
	Note that $I_{S_n}$ contains no polynomials of degree $1$. 
	From Corollary~\ref{cor: vertices of state pol}, we know that 
	\[
		\Newt(U_n) = \conv\{(\pi,\pi^c) \mid \pi \in \symm_n\} -  (1,\dots,1).
	\]
	Proposition~\ref{prop: newt is permuto} tells us that $\Newt(U_n)$ is unimodularly equivalent to $\Pi_n$, so choosing $w \in N_{\Newt(U_n)}(\{v\})$ can be done by choosing $w$ in an appropriate Weyl chamber and applying a unimodular transformation.
	As a result, the point $\init_{\prec_w}(I_{S_n})_2$ is exactly $\sum \Newt(\init_{\prec_w}(t^a - t^b))$ where the sum is over all $t^a - t^b \in U_n$.
	The resulting points are $(\pi,\pi^c) - (1\dots,1)$, which have already been computed as vertices of $\Newt(U_n)$.
	So, these are all the vertices of the state polytope of $I_{S_n}$.
	Therefore, the two are unimodularly equivalent. 
\end{proof}


\subsection{For the code $P(2_n)$}

Recall that a matrix $A \in \ZZ^{n \times k}$ of rank $n$ is called \emph{unimodular} if all of its maximal minors are a $0$ or $\pm 1$.  
More strongly, $A$ is \emph{totally unimodular} if all of its minors are $0$ or $\pm 1$.
So, if $A$ has rank $n$ and is totally unimodular, then it is also unimodular.
A class of totally unimodular matrices are those with the \emph{consecutive ones property}: $0/1$ matrices such that, in each row, all $1$s appear consecutively.
The fact that such matrices are totally unimodular is a standard linear algebra exercise. 
This is what will allow us to compute the universal Gr\"obner basis of $P(2_n)$.

\begin{lemma}[{\cite[Proposition 8.11]{sturmfels}}]\label{lem: unimod}
	If $A$ is a unimodular matrix, then its Graver basis and universal Gr\"obner basis coincide. 
\end{lemma}

To help us describe the universal Gr\"obner basis of $P(2_n)$, we will use the following notation.
If $t^u \in \CC[t_1,\dots,t_m]$, the \emph{support} of $t^u$ is $\supp(t^u) = \{i \in [m] \mid t_i \text{ divides } t^u\}$.
Also, recall that the codewords in $P(2_n)$ are the columns of a matrix $M_n$.

\begin{prop}\label{prop: UGB for P}
	Let
	\[
		V'_n = \{t_{3i+1}t_{3i+3} - t_{3i+2}t_{3n+1} \mid i = 0, \dots, n-1\}
	\]
	and
	\[
		V''_n = \{t_{3i+1}t_{3i+3}t_{3j+2} - t_{3j+1}t_{3j+3}t_{3i+2}  \mid 0 \leq i < j \leq n-1\}.
	\]
	The universal Gr\"obner basis $V_n$ for $I_{P(2_n)}$ is $V_n = V'_n \cup V''_n$.
\end{prop}

\begin{proof}
	By construction, $M_n$ has the consecutive ones property, so $M_n$ is totally unimodular.
	Furthermore, the transpose of $M_n$ clearly has rank $2n+1$, so $M_n$ also has rank $2n+1$.
	Therefore, $M_n$ is unimodular, and by Lemma~\ref{lem: unimod}, we only need to find the Graver basis of $I_{P(2_n)}$.
	It is straightforward to check that each binomial in $V_n$ is primitive and in the toric ideal, so these are each in the Graver basis.
	All that remains is to show that there are no other primitive binomials in $I_{P(2_n)}$.
	
	Let $t^u - t^v$ be a primitive binomial in $I_{P(2_n)}$. 
	If $\deg t^u = 2$, then 
	\[
		\supp(t^u) \in \{3i+1,3i+2,3i+3,3n+1\}
	\]
	for some $i = 0,\dots, n-1$ since, if this were not the case, then $t^u = t^v$.
	By checking cases, it follows that $\pm(t^u - t^v) \in V'_n \subseteq V_n$. 
	An analogous argument also shows that if $\deg t^u = 3$, then $\pm(t^u - t^v) \in V''_n \subseteq V_n$.
	
	Suppose $\deg t^u \geq 4$. 
	It must first be true that $|\supp(t^u)| > 1$; if $|\supp(t^u)| = 2$, then $\supp(t^u) \subseteq \{3i+1,3i+2,3i+3,3n+1\}$, and $\supp(t^v) = \{3i+1,3i+2,3i+3,3n+1\} \setminus \supp(t^u)$.
	So, $t^u$ is divisible by some $t_{\alpha}t_{\beta}$ and $t^v$ is divisible by some $t_{\gamma}t_{\delta}$ for which $t_\alpha t_\beta - t_\gamma t_\delta \in V_n$.
	This contradicts $t^u-t^v$ being primitive.
	
	If $|\supp(t^u)| > 2$, then there are two possibilities, at least one of which must occur via the pigeonhole principle: for the first possibility, there may exist monomials $t_\alpha t_\beta$ and $t_\gamma t_\delta$ as in the previous paragraph, which again implies $t^u - t^v$ is not primitive.
	For the second possibility, there may exist a (squarefree) monomial $t_{\alpha}t_{\beta}t_{\gamma}$ dividing $t^u$ for which $\{\alpha,\beta,\gamma\} = \{3i+1,3i+3,3j+2\}$ for some $i \neq j$.
	If this happens, and if $t^u - t^v$ does not satisfy the first case, then there is a monomial $t_at_bt_c$ dividing $t^v$ such that $\{a,b,c\} = \{3i+2,3j+1,3j+3\}$ for the same choice of $i,j$.
	Once again, this implies $t^u - t^v$ is not primitive.
	Therefore, $t^u - t^v$ must have degree at most $3$, in which case $\pm(t^u - t^v) \in V_n$, and $V_n$ is the universal Gr\"obner basis for $I_{P(2_n)}$.
\end{proof}

Although we could perform Algorithm~\ref{alg: state pol} as in the previous section, it would be more difficult to do since the universal Gr\"obner basis now contains cubic binomials.
However, all is not lost.

Let $A$ be a set of $m$ generators for $\NN A \subseteq \ZZ^n$.
A vector $b \in \NN A$ is a \emph{Gr\"obner degree} if there is some binomial $t^u - t^v$ in the universal Gr\"obner basis of $I_A$ such that $\pi(t^u) = \pi(t^v) = x^b$. 
If $b$ is a Gr\"obner degree, then the polytope
\[
	\conv\{a \in \RR^m \mid t^a \in \pi^{-1}(x^b)\}
\]
is called a \emph{Gr\"obner fiber}. 
A crucial result using Gr\"obner fibers that will simplify the work we need to do is the following.

\begin{thm}[{\cite[Theorem 7.15]{sturmfels}}]
	The Minkowski sum of all Gr\"obner fibers is a state polytope for $I_A$.
\end{thm}

\begin{thm}
	Let $n \geq 1$, and set 
	\[\begin{aligned}
		Q_n =& \sum_{i = 0}^{n-1} \conv\{e_{3i+1}+e_{3i+3}, e_{3i+2}+e_{3n+1}\} \\
		&+ \sum_{0 \leq i < j \leq n-1} \conv\{e_{3i+1}+e_{3i+3}+e_{3j+2}, e_{3i+2} + e_{3j+1}+e_{3j+3}, e_{3i+2}+e_{3j+2}+e_{3n+1}\}.
	\end{aligned}\]
	Then $Q_n$ is a state polytope for $P(2_n)$.
\end{thm}

\begin{proof}
	Let $t_{3i+1}t_{3i+3} - t_{3i+2}t_{3n+1}$ be an element of $V_n$ for some $i = 0,\dots,n-1$.
	Then $\pi(t_{3i+1}t_{3i+3}) = x_{2i+1}x_{2i+2}x_{2n+1}^2$. 
	Note that the only columns of $M_n$ having any nonzero entries in either of the coordinates $2i+1$ or $2i+2$ are columns $3i+1,3i+2,$ and $3i+3$.
	Thus, $\pi^{-1}(x_{2i+1}x_{2i+2}x_{2n+1}^2)$ must consist of quadratic monomials involving only the variables $t_{3i+1},t_{3i+3}, t_{3i+2},$ and $t_{3n+1}$.
	It is then clear that there are only two possibilities, so that
	\[
		\pi^{-1}(x_{2i+1}x_{2i+2}x_{2n+1}^2) = \{t_{3i+1}t_{3i+3}, t_{3i+2}t_{3n+1}\}.
	\]
	Therefore, the corresponding Gr\"obner fiber is the line segment $\conv\{e_{3i+1} + e_{3i+3}, e_{3i+2}+e_{3n+1}\}$.
	
	Now let $g = t_{3i+1}t_{3i+3}t_{3j+2} - t_{3j+1}t_{3j+3}t_{3i+2}$ be an element of $V_n$ for some choice of $i < j$. 
	This time,
	\[
		\pi(t_{3i+1}t_{3i+3}t_{3j+2}) = x_{2i+1}x_{2i+2}x_{2j+1}x_{2j+2}.
	\]
	There are six columns of $M_n$ which have nonzero entries in coordinates $2i+1,2i+2,2j+1$, and $2j+2$, which are exactly the columns corresponding to indices of variables in $g$.
	Our task, then, is to determine all cubics in those variables, along with $t_{3n+1}$, whose image is $x_{2i+1}x_{2i+2}x_{2j+1}x_{2j+2}$.
	Computing the preimage of the right hand side we find
	\[
		\pi^{-1}(x_{2i+1}x_{2i+2}x_{2j+1}x_{2j+2}) = \{t_{3i+1}t_{3i+3}t_{3j+2}, t_{3j+1}t_{3j+3}t_{3i+2}, t_{3i+2}t_{3j+2}t_{3n+1}\}.
	\]
	Therefore, the corresponding Gr\"obner fiber is the triangle
	\[
		\conv\{e_{3i+1}+e_{3i+3}+e_{3j+2}, e_{3j+1}+e_{3j+3}+e_{3i+2}, e_{3i+2}+e_{3j+2}+e_{3n+1}\}.
	\]
\end{proof}

Recall that if $n$ is a positive integer, then $\binom{[n]}{k}$ denotes the collection of $k$-subsets of $[n]$.
For ease of notation, we will henceforth let
\[
	\calI_n = \left\{S \cup \{n+1\} \mid S \in \binom{[n]}{1} \cup \binom{[n]}{2}\right\}.
\]
Further, if $S \subseteq [n]$, let $\Delta_S = \conv\{e_i \in \RR^n \mid i \in S\}$.

\begin{cor}\label{cor: minkowski sum}
	The state polytope for $P(2_n)$ is unimodularly equivalent to
	\[
		\overline Q_n = \sum_{S \in \calI_n} \Delta_S.
	\]
\end{cor}

\begin{proof}
	Each vertex of $Q_n$ is the sum of vertices of the Minkowski summands.
	Note that each summand satisfies the equation
	\[
		x_{3i+1} - x_{3i+3} = 0
	\]
	for all $i=0,\dots,n-1$.
	
	Next notice that for each Minkowski summand of $Q_n$, whenever one vertex has a nonzero entry in coordinate $3i+1$, all other vertices are zero in coordinate $3i+1$ and contain a one in coordinate $3i+2$.
	Since there are $n$ summands for which one of the vertices has a one in coordinate $3i+1$, we know $Q_n$ satisfies the equations
	\[
		x_{3i+1} + x_{3i+2} = n
	\]
	for all $i=0,\dots,n-1$.
	
	These observations allow us to project $Q_n$ onto coordinates $3i+1$ for $i=0,\dots,n$ to obtain a unimodularly equivalent polytope, $\overline Q_n$.
	Formally, we have shown $Q_n$ is unimodularly equivalent to
	\[
		\overline Q_n = \sum_{i = 0}^{n-1} \conv\{e_{3i+1}, e_{3n+1}\} + \sum_{0 \leq i < j \leq n-1} \conv\{e_{3i+1},e_{3j+2},e_{3n+1}\},
	\]
	which is identical to the sum claimed with only a change in notation.
\end{proof}

We can say more about $\overline Q_n$ than just this implicit expression.
Indeed, its vertices can be explicitly computed, and these are what allow us to identify the number of initial ideals of $I_{P(2_n)}$.
First, we prove a lemma to help us along the way. 
To simply notation, if $S \in \calI_n$, let $H^{\geq}_S$ denote the closed halfspace
\[
	H^{\geq}_S = \left\{(x_1,\dots,x_n) \in \RR^n \mid \sum_{i \in S} x_i \geq \binom{|S|}{2}\right\}.
\]
Define the open halfspace $H^{>}_S$ and the hyperplane $H^{=}_S$ analogously.

\begin{lemma}\label{lem: Qn hyperplanes}
	For all $n$, $\overline Q_n$ is the intersection
	\[
		\overline Q_n = H_{[n]}^= \cap \left(\bigcap_{S \in \calI_n} H^{\geq}_S\right).
	\]
\end{lemma}

\begin{proof}
	By Corollary~\ref{cor: minkowski sum}, $\overline Q_n = \sum_S y_S\Delta_S$ where $y_S = 1$ if $S \in \calI_n$ and $y_S = 0$ otherwise.
	So, by \cite[Proposition 6.3]{Postnikov2009}, 
	\[
		\overline Q_n = \left\{(x_1,\dots,x_{n+1}) \in \RR^{n+1} \mid \sum_{i=1}^{n+1} x_i = z_{[n+1]},\, \sum_{i \in S} x_i \geq z_S \text{ for subsets } S \subseteq [n+1]\right\}
	\]
	where
	\[
		z_S = \sum_{J \subseteq S} y_J.
	\]
	By construction, $z_S = 0$ for any $S$ not containing $n+1$, and $z_S = |S| - 1 + \binom{|S|-1}{2} = \binom{|S|}{2}$ for any $S$ containing $n+1$.
	Thus, any halfspace of the form $H^{\geq}_S$ for $S \subseteq [n]$ a non-singleton is redundant, being implied by the intersection of the halfspaces $H^{\geq}_{\{i\}}$ for $i \in S$.
\end{proof}

Now let $\calF \subseteq 2^{[n]}$ be a collection of nonempty subsets.
We call $\calF$ a \emph{building set} on $[n]$ if the following conditions are satisfied:
\begin{enumerate}
	\item If $F_1,F_2 \in \calF$ and $F_1 \cap F_2 \neq \emptyset$, then $F_1 \cup F_2 \in \calF$; and
	\item $\calF$ contains $\{i\}$ for all $i \in [n]$.
\end{enumerate}

Any collection of subsets $\calF \subseteq 2^{[n]}$ can be extended to a building set.
This extension can be done in a way such that the resulting building set, considered as a simplicial complex on $[n]$, is minimal with respect to inclusion and is the unique building set with this property \cite[Lemma 3.10]{FeichtnerSturmfels}.
We call the building set constructed in this manner the \emph{building closure} of $\calF$ and denoted it by $\widehat{\calF}$.

To construct $\widehat\calF$, first consider a fixed nonempty subset $S$ of $[n]$.
Let $\calF_{\leq S}$ denote the subsets of $S$ that are also elements of $\calF$. 
Then $\widehat\calF$ is the set of all subsets $S$ of $[n]$ for which $|S| = 1$ or $\calF_{\leq S}$ is a connected simplicial complex. 

\begin{lemma}\label{lem: closure}
	The building closure of $\calI_n$ is 
	\[
		\widehat{\calI}_n = \binom{[n]}{1} \cup \left(\bigcup_{J \subseteq [n]} J \cup \{n+1\}\right).
	\]
\end{lemma}

\begin{proof}
	Let $\calJ$ denote the claimed building closure of $\calI_n$.
	It is clear that $\{i\} \in \calJ$ for each $i \in [n+1]$.
	So, suppose $J \subseteq [n+1]$.
	If $n+1 \notin J$, then $(\calI_n)_{\leq J}$ is empty, since every element of $\calI_n$ contains $n+1$.
	In particular, in this case, $(\calI_n)_{\leq J}$ is not a simplicial complex, so $J \notin \calJ$.
	
	Now suppose $n+1 \in J$.
	If $|J| = \{n+1\}$ then $(\calI_n)_{\leq J} = \{\emptyset, \{n\}\}$ which is clearly connected.
	Otherwise, $(\calI_n)_{\leq J}$ consists of pairs $\{j,n+1\}$ and its subsets and/or of triples $\{i,j,n+1\}$ and its subsets.
	Since all of the maximal elements of these contain $n+1$, $(\calI_n)_{\leq J}$ is connected, $J \in \calJ$.
	Therefore, $\calJ = \widehat{\calI}_n$.
\end{proof}

When $\calF$ itself is a building set, then there is a large number of results known about the combinatorial structure of $\sum_{F \in \calF} \Delta_F$.
However, since $\calI_n$ is not a building set, we have to prove the desired conclusions directly.

Let $\calF$ be a building set.
A subset $N$ of $\calF$ is a \emph{nested set} if the following conditions hold:
\begin{enumerate}
	\item For $I,J \in N$, either $I \subseteq J$, $J \subseteq I$, or $I \cap J = \emptyset$;
	\item For any disjoint collection $J_1,\dots,J_k \in N$, $k \geq 2$, we have $J_1 \cup \cdots \cup J_k \notin \calF$; and
	\item $N$ contains all inclusion-maximal elements of $\calF$.
\end{enumerate}
The collection of all nested sets in $\calF$ is called the \emph{nested set complex} $\calN(\calF)$.

\begin{lemma}\label{lem: maximal nested sets}
	The inclusion-maximal elements of $\calN(\widehat{\calI}_n)$ are of the form $\{I_1,\dots,I_{n+1}\}$ such that $I_1, \dots, I_k$ are $k$ (possibly zero) disjoint singletons of $[n]$, $|I_{k+1} \setminus (I_1 \cup \cdots \cup I_k)| = \{n+1\}$, and $|I_{j+1} \setminus I_j| = 1$ for all $j > k$.
\end{lemma}

\begin{proof}
	Let $N \in \calN(\widehat\calI_n)$ be maximal.
	For our first case, suppose $I \in N$ is a singleton from $[n]$.
	If $J \in N$ and $I \cap J = \emptyset$, then by condition $2$ for nested sets, $I \cup J \notin \widehat\calI_n$.
	But $J \in \widehat\calI_n$, so that must mean $J$ is a singleton of $[n] \setminus I$. 
	
	Let $I_1,\dots, I_k$ be the distinct singletons in $N$, which we are assuming must be subsets of $[n]$.
	Since $k \leq n$, there must be some $J \in N$ of size at least $2$. 
	If $J$ is disjoint from any of $I_1,\dots,I_k$, then the previous argument implies that $J$ is also a singleton, which is a contradiction.
	So, $I_1,\dots,I_k \subseteq J$.
	Since we need $J \in \widehat\calI_n$, this forces $n+1 \in J$ as well.
	
	Observe that by condition $2$ of nested sets and the pigeonhole principle, any element of $\calN(\widehat\calI_n)$ can have at most $n+1$ elements.	
	Thus if $N$ does not satisfy condition $1$ of the stated lemma, there are two distinct elements $S,T \in N$ for which $|S|, |T| > 1$ and $S \cap T \neq \emptyset$.
	If $|S| = |T|$, this contradicts condition $1$ of $N$ being a nested set, since neither $S$ nor $T$ can be a subset of the other and their intersection is nontrivial.
	Therefore, no such $S$ and $T$ can exist.
	Thus, either $S \subseteq T$ or $T \subseteq S$.

	The above implies that if $N = \{I_1,\dots,I_{n+1}\}$, then the $I_j$ can be relabeled so that $|I_1| = \cdots = |I_k| = 1$, $I_{k+1} = I_1 \cup \cdots I_k \cup \{n+1\}$, and $|I_{j+1} \setminus I_j| = 1$ for all $j > k+1$. 	
	By construction and Lemma~\ref{lem: closure}, $I_j \in \widehat\calI_n$ for each $j=1,\dots,n+1$, and conditions $1$ and $2$ of nested sets are satisfied.
	For the final condition, note that $[n+1]$ is the unique inclusion-maximal element of $\widehat\calI_n$, and is equal to $I_{n+1}$.
	So, each set described in condition $1$ of the statement of the lemma is a maximal element of $\calN(\widehat\calI_n)$.
	
	Finally, suppose $I_1 = \{n+1\}$.
	If there is another singleton $I \in N$, then $I_1 \cup I \in \widehat\calI_n$, so the only singleton of $N$ is $I_1$. 
	The conclusion of the lemma follows, continuing the argument from the third paragraph of this proof.
\end{proof}

\begin{thm}\label{thm: Qn simple}
	For all $n$, $\overline Q_n$ is simple and $n$-dimensional.
\end{thm}

\begin{proof}
	That $\overline Q_n$ is $n$-dimensional follows immediately from \cite[Remark 3.11]{FeichtnerSturmfels}.
	Further, by \cite[Corollary 3.13]{FeichtnerSturmfels}, the set of halfspaces given in Lemma~\ref{lem: Qn hyperplanes} form an irredundant halfspace description of $\overline Q_n$.
	However, since we also have the hyperplane $H^=_{[n+1]}$, we know that the halfspace defined by $H^{\geq}_{[n+1]}$ can be replaced by $H^=_{[n+1]}$.
	So, we can ignore $H^{\geq}_{[n+1]}$ when showing that each vertex of $\overline Q_n$ is the intersection of exactly $n$ hyperplanes.
	
	Let $N \in \widehat\calI_n$. 
	By Lemma~\ref{lem: maximal nested sets}, $N$ is of the form $\{I_1,\dots,I_{n+1}\}$ where $I_1 = \{i_1\},\dots,I_k = \{i_k\}$ are distinct singletons, and $I_j = \{i_1, \dots, i_j\}$ for $j > k$ and $i_{k+1} = n+1$.
	The normal vectors to $H^=_{I_1},\dots,H^=_{I_n}$ are linearly independent, so the intersection of the hyperplanes form an affine line in $\RR^{n+1}$.
	The hyperplane $H^=_{[n+1]}$ intersects this line at the point $(v_1,\dots,v_{n+1})$, where
	\[
		v_r = \begin{cases}
				n+1-r & \text{ if } \{r\} = I_{j+1} \setminus I_j \text{ for some } j > k+1 \\
				\binom{k}{2} & \text{ if } r = k+1\\
				0 & \text{ if } r \leq k
			\end{cases}
	\] 
	Note in particular that $v$ lies in $H^>_I$ for all other $I \notin N$: if $I = \{i\}$ is a singleton not in $N$, then by construction, $v_i > 0$;
	and if $n \in I$ and $|I| > 1$, then there is some $j \in I$ for which $\{j\} \notin N$ and $v_j > 0$.
	So, by construction, $\sum_I v_i > \binom{|I|}{2}$.
	
	It remains to show that no other intersection of $n$ hyperplanes $H^=_I$, together with $H^=_{[n+1]}$, form a vertex of $\overline Q_n$.
	First consider $H^=_{\{i\}}$ and $H^=_{\{n+1\}}$ for some $i < n+1$.
	The points lying on both hyperplanes satisfy $x_i = x_{n+1} = 0$, but no such point can satisfy $x_i + x_{n+1} \geq 1$.
	So, these two hyperplanes cannot be used together to define a vertex.
	
	Next, suppose $v$ is defined by hyperplanes $H^=_{I_1},\dots,H^=_{I_r}$ and $H^=_{[n+1]}$ for some $r \geq n$ such that $|I_\alpha|, |I_\beta| > 1$ and $\{I_\alpha,I_\beta\}$ is not in any nested set in $\widehat\calI_n$.  
	Without loss of generality we can assume $\alpha = 1$ and $\beta = 2$. 
	We then have $|I_1 \setminus I_2|, |I_2 \setminus I_1| > 0$.
	In particular, $|I_1|, |I_2| > |I_1 \cap I_2|$. 
	
	Consider the hyperplanes $H^=_{I_1}, H^=_{I_2}, H^=_{I_1 \cup I_2}$, and $H^=_{I_1 \cap I_2}$.
	Using the fact that $|I_1 \cup I_2| = |I_1| + |I_2| - |I_1 \cap I_2|$ twice, we get the sequence of (in)equalities
	\[\begin{aligned}
		\binom{|I_1| + |I_2| - |I_1 \cap I_2|}{2} &= \binom{|I_1 \cup I_2|}{2}\\
			&\leq \sum_{i \in I_1 \cup I_2} x_i \\
			&= \binom{|I_1|}{2} + \binom{|I_2|}{2} - \sum_{i \in I_1 \cap I_2} x_i \\
			&\leq \binom{|I_1|}{2} + \binom{|I_2|}{2} - \binom{|I_1 \cap I_2|}{2}.
	\end{aligned}\]
	After routine algebraic manipulation, we deduce that
	\[
		(|I_1| - |I_1 \cap I_2|)(|I_2| - |I_1 \cap I_2|) \leq 0.
	\]
	However, this is impossible, since $|I_1|, |I_2| > |I_1 \cap I_2|$.
	This completes the proof. 
\end{proof}

As part of the previous proof, we explicitly found the vertices of $\overline Q_n$.
We will use the following notation to make the description simpler to work with.
Let $\pi \in \symm_n$ be any permutation.
We will treat $\pi$ both as a word $\pi = \pi_1\dots\pi_n$ and as a vector $(\pi_1,\dots,\pi_n)$.
For $i = 0,\dots,n$, let 
\[
	\tau_i(\pi) = \left(\chi_{[n]\setminus[i]}(\pi_1),\dots,\chi_{[n]\setminus[i]}(\pi_n),\binom{i+1}{2}\right)
\]
where $\chi$ is the indicator function.
In words, $\tau_i(\pi)$ replaces each instance of $\pi_j \leq i$ with $0$, and appends their sum to the end of the vector. 

\begin{cor}
	The vertices of $\overline Q_n$ are
	\[
		\{\tau_i(\pi) \mid i = 0,\dots,n,\, \pi \in \symm_n\}.
	\]
	Consequently, there are exactly
	\[
		\sum_{i=0}^n \binom{n}{i}i! = \sum_{i=0}^n \frac{n!}{i!}
	\]
	distinct initial ideals of $I_{P(2_n)}$.
\end{cor}

\begin{proof}
	The description of the vertices is a restatement of the vertices computed in the proof of Theorem~\ref{thm: Qn simple}.
	The enumeration of the vertices, and therefore the initial ideals of $P(2_n)$, is a routine counting argument and is therefore omitted.
\end{proof}

We end this section with a brief, but interesting, connection to graph-associahedra.
Let $\st_n$ denote the star graph on $[n+1]$ with edges $E(\st_n) = \{i,n+1\}$ for each $i=1,\dots,n$.

\begin{defn}
	Let $n$ be a positive integer.
	The $n^{th}$ \emph{stellohedron} is the polytope
	\[
		\stell_n = \sum_{S \in E(\st_n)} \Delta_S.
	\]
\end{defn}

Two polytopes are said to be \emph{combinatorially equivalent} if their poset of faces, ordered by inclusion, are isomorphic.
For example, all convex quadrilaterals in the plane are combinatorially equivalent.

\begin{cor}
	For all $n$, $\overline Q_n$ is combinatorially equivalent to the stellohedron $\stell_n$.
\end{cor}

\begin{proof}
	By Lemma~\ref{lem: Qn hyperplanes}, \cite[Section 10.4]{PostnikovReinerWilliams}, and \cite[Corollary 3.13]{FeichtnerSturmfels}, the set of vectors normal to the facets of $\overline Q_n$ are the same as those normal to the facets of $\stell_n$. 
	So, their normal fans have the same rays. 
	Since $\overline Q_n$ is a Minkowski summand of $\stell_n$, the normal fan of $\stell_n$ refines the normal fan of $\overline Q_n$ \cite[Proposition 1.2 (3)]{BilleraFillimanSturmfels}.
	From Theorem~\ref{thm: Qn simple} we know that $\overline Q_n$ is simple, so $N(\overline Q_n)$ is simplicial.
	Since $N(\stell_n)$ is also simplicial, the only way for $N(\stell_n)$ to properly refine $N(\overline Q_n)$ is if there is a cone in the former which is not in the latter.	
	However, each ray in the normal fan of a simple polytope is a face of all cones in which it is contained.
	So, $N(\overline Q_n)$ cannot be properly refined, hence the two normal fans are the same.
	Therefore, $\overline Q_n$ and $\stell_n$ have the same combinatorial structure.
\end{proof}


\section{Experimental Data}\label{sec: experimental data}

In the previous section, we showed that the state polytopes for certain combinatorial neural codes are unimodularly equivalent to well-known polytopes.
The second class was a special case of $P(l_1,\dots,l_k)$.
It would clearly be of interest to determine which choices of $l_1,\dots,l_k$ result in state polytopes with interesting properties.

One natural choice to make is $P(l,0,\dots,0)$ where we allow any $l > 0$. 
A corresponding Euler diagram consists of curves $\lambda_1,\dots,\lambda_{l+1}$ such that $\lambda_i$ and $\lambda_j$ intersect if and only if $|i - j| = 1$.
It is not entirely practical to use the techniques from previous sections on this code, even though the matrix whose columns are the codewords of the code has the consecutive ones property, and is therefore totally unimodular.
Indeed, experimental data in CoCoA \cite{CoCoA} suggests that the Minkowski summands of the state polytope become significantly more complicated; when $n=4$, the universal Gr\"obner basis consists of nine quadratics, eleven cubics, and three quartics.

\begin{problem}
	Characterize the following sets for all choices of $\ell \in \ZZ^k_{\geq 0}$:
	\begin{enumerate}
		\item the universal Gr\"obner basis of $P(\ell)$;
		\item the state polytope of $P(\ell)$; and
		\item all initial ideals of $P(\ell)$.
	\end{enumerate}
\end{problem}

\begin{table}
\[\begin{array}{|c|c|c|}\hline
\text{quadratics} & \text{cubics} & \text{quartics} \\ \hline
t_9t_{11} -t_{10}t_{12} &   t_5t_7t_{10} -t_6t_9t_{11}&  t_1t_3t_6t_{10} -t_2t_5t_8t_{11}\\
  t_7t_{10} -t_8t_{11}&  t_3t_5t_8 -t_4t_7t_9&  t_1t_4t_7t_{10} -t_2t_6t_9t_{11}\\
  t_7t_9 -t_8t_{12}&  t_3t_5t_{10} -t_4t_9t_{11}&  t_1t_4t_8t_{11} -t_2t_6t_{10}t_{12} \\
  t_5t_8 -t_6t_9&  t_1t_3t_8 -t_2t_7t_9&\\
  t_5t_7 -t_6t_{12}&  t_1t_3t_6 -t_2t_5t_7&\\
  t_3t_6 -t_4t_7&  t_1t_3t_{10} -t_2t_9t_{11}&\\
  t_3t_5 -t_4t_{12}&  t_5t_8t_{11} -t_6t_{10}t_{12}&\\
  t_1t_4 -t_2t_5&  t_3t_6t_{10} -t_4t_8t_{11}&\\
  t_1t_3 -t_2t_{12}&  t_3t_6t_9 -t_4t_8t_{12}&\\
&  t_1t_4t_8 -t_2t_6t_9&\\
&  t_1t_4t_7 -t_2t_6t_{12}&\\
\hline
\end{array}\]
\caption{The binomials of the universal Gr\"obner basis for $P(5,0,\dots,0)$, arranged by degree.}
\end{table}

We conclude with a conjecture based on data in Macaulay2 \cite{M2}.
Given a polytope $P$, let its $k^{th}$ face number $f_k$ denote the number of faces of $P$ of dimension $k$.

\begin{conj}
	Let $\ell = (l,0,\dots,0) \in \NN^n$ for some $l$.
	The state polytope of $P(\ell)$ has face numbers
	\[
		f_k = \binom{n-1}{k}\binom{2(n-1)}{n-1}.
	\]
	Combinatorially, $f_k$ is the number of Delannoy paths from $(0,0)$ to $(n-1,n-1)$ with $k$ diagonal steps.
\end{conj}

\bibliographystyle{plain}
\bibliography{references}

\end{document}